\documentclass[12pt]{amsart}

\usepackage[matrix,arrow,curve,frame]{xy}
\usepackage{amsmath,amsthm,amssymb,enumerate}
\usepackage{latexsym}
\usepackage{amscd}
\usepackage[colorlinks=true]{hyperref}
\usepackage{euscript}
\usepackage{url}

\newtheorem{thm}[subsection]{Theorem}
\newtheorem{defn}[subsection]{Definition}

\newtheorem{claim}[subsection]{Claim}
\newtheorem{corr}[subsection]{Corollary}
\newtheorem{lemma}[subsection]{Lemma}

\newtheorem{remark}[subsection]{Remark}

\theoremstyle{definition}
\newtheorem{example}[subsection]{Example}

\newcommand{\Z}{\mathbb Z}
\newcommand{\F}{\mathbb F}

\newcommand{\C}{\mathbb C}

\DeclareMathOperator{\Q}{\mathbb{Q}}

\newfont{\german}{eufm10}

\newcommand\qu{/\kern-.7ex/}

\begin{document}
\pagestyle{plain}

\title
{Soergel Bimodules, the Steenrod Algebra and Triply Graded Link homology}
\author{Nitu Kitchloo}
\address{Department of Mathematics, Johns Hopkins University, Baltimore, USA}
\email{nitu@math.jhu.edu}
\thanks{Nitu Kitchloo is supported by the Simons Foundation and the Max Planck Institute for Mathematics.}
\maketitle

\begin{abstract}
Soergel bimodules and their Hochschild homology are known to be important in the context of link homology. In this article we observe that Soergel bimodules may be naturally identified as the cohomology of well-defined objects in the category of spectra. This allows us to define forms of Soergel bimodules over the integers after inverting $2$, indexed by elements in Braid groups associated to compact Lie groups of adjoint type. Reducing the Soergel bimodules modulo an odd prime, we endow the bimodules with an action of the reduced Steenrod algebra. This action extends to an action of the reduced Steenrod algebra on the Hochschild homology of Soergel bimodules over the field $\F_p$ for odd primes $p$. In the special case of the $n$-stranded Braid group, our results allow us to define the triply graded link homology over any ring where $2$ is invertible, and deduce that the reduced Steenrod algebra acts on triply graded link homology over the field $\F_p$ for odd primes $p$.
\end{abstract}

\tableofcontents

\section{Introduction}

\medskip
\noindent
In \cite{K} Khovanov studies a class of complexes of bimodules, known as Soergel bimodules, over a polynomial algebra $R$ in variables $X_1, \ldots, X_r$ defined over a field of characteristic zero. These complexes are indexed by elements of the braid group on $r+1$-strands, and are defined if one is given a presentation of a braid word as a product of elementary braids and their inverses (see section \ref{one}). It is shown in \cite{K} that the Hochschild homology of the bimodules in this complex, with coefficients in the bimodule $R$, is an invariant of the link given by closing the braid, and is closely related to link homology up to degree shifts. In this note we extend most of this structure to arbitrary rings where $2$ is invertible (see definition \ref{TGLH}), and also endow link homology with the action of the reduced Steenrod algebra when working over fields $\F_p$ for odd primes $p$ (see theorem \ref{Linkinv}). Integral lifts of link homology have also been considered in the past by other authors \cite{Ka}. 

\medskip
\noindent
More generally, Soergel bimodules can be defined for arbitrary compact connected Lie groups $G$ of adjoint type (i.e. center free), with the ones considered above corresponding to $G = PU(r+1)$. Working with general groups $G$ of adjoint type and modulo an odd $p$, we enrich Soergel bimodules with an action of the reduced Steenrod algebra. This is achieved by describing Soergel bimodules as the (mod $p$) cohomology of certain topological constructions. We elaborate on these constructions below. The action of the reduced Steenrod algebra then extends to the Hochschild homology of Soergel bimodules (which in the case of $G = PU(r+1)$ gives rise to link homology as mentioned earlier). 

\medskip
\noindent
We were directly influenced by \cite{KR}, and the author is indebted to M. Khovanov and L. Rozansky for their interest. We thank Vitaly Lorman for a thorough reading and helpful feedback,  and Jack Morava for sharing his insights. The author would like to acknowledge the Simons Foundation and the Max Planck Institute for Mathematics for their support during the period when this paper was finalized. 

\subsection{The basic argument} 

\medskip
\noindent

\smallskip
\noindent
The fundamental observation in this document is that all algebraic constructions made with Soergel bimodules in \cite{K} can be enriched in a triangulated category known as the homotopy category of spectra, which is a natural generalization of the homotopy category of topological spaces (see section \ref{FGL} for more details on the spectra relevant to us).\footnote{One can find other enrichments of the category of Soergel bimodules in various topological categories (see \cite{WW}).} Furthermore, these spectra have (singular) cohomology that is evenly graded and torsion free away from the prime $2$. More precisely, given a braid word written in terms of elementary reflections of the form $\tilde{\sigma} = \sigma_{i_1}^{\epsilon_1} \, \sigma_{i_2}^{\epsilon_2} \cdots \sigma_{i_k}^{\epsilon_k}$, with $\epsilon_i = \{-1,1\}$, consider the following constructions of the complex $F_\bullet(\tilde{\sigma})$ of Soergel bimodules in terms of two-term complexes $F(\sigma_i^{\epsilon_i})$ defined in \cite{K}:
\[ F_\bullet(\tilde{\sigma}) = F(\sigma_{i_1}^{\epsilon_1}) \otimes_R F(\sigma_{i_2}^{\epsilon_2}) \otimes_R \cdots \otimes_R F(\sigma_{i_k}^{\epsilon_k}), \]
where we recall that $R$ is a polynomial algebra of a finite number of variables. We observe that the terms of the complex $F_\bullet(\tilde{\sigma})$ are free $R$-modules which can in fact be defined over the ground ring $\Z[\frac{1}{2}]$, the integers with $2$ being inverted, and can be canonically identified with the singular cohomology (with $\Z[\frac{1}{2}]$-coefficients) of a complex of Thom spectra (of related vector bundles). Under this identification, $R$ corresponds to the cohomology of the space $BT$, where $T$ is a torus of finite rank $r$, with classifying space denoted by $BT$. We will also show that the complex $F_\bullet(\tilde{\sigma})$ depends only on the underlying braid group element and not on the presentation (see theorem \ref{quasi}). In addition, this independence of presentation is achieved by a zig-zag of maps induced by the Bott-Samelson resolutions. Working modulo an odd prime $p$, the complex $F_\bullet(\tilde{\sigma})$ therefore supports the action of the reduced Steenrod algebra $\mathcal{A}_p$ (definition \ref{redSteenrod}). Next, we consider the Koszul resolution (see definition \ref{CHH}) that computes the Hochschild homology:
\[ HH_*(F_\bullet(\tilde{\sigma}), R). \]
Working over the field $\F_p$ for an odd prime $p$, we then extend the action of the reduced Steenrod algebra to the Hochschild homology. Specializing to the case of the projective unitary groups $G = PU(r+1)$, we decude that the reduced Steenrod algebra acts on triply-graded link homology (see theorem \ref{Linkinv}). 

\smallskip
\subsection{Organization of the Paper}

\smallskip
\noindent

\medskip
\noindent
Before we begin with details, let us describe how this paper is organized, and justify the value of each section. We begin in section \ref{one} by introducig the objects in the homotopy category of spectra that give rise to the complexes of bimodules once we take cohomology. This enrichment makes it evident that {\em all induced maps in any cohomology theory are natural with respect to cohomology operations}. In later sections, we will apply the cohomology operations known as the Steenrod reduced power operations to the Hochschild bicomplex for a braid-word. This bicomplex is defined at the end of section \ref{one}. The following section (section \ref{FGL}) introduces the cohomology operations of interest, and supplies the necessary background on Thom spectra, which are the spectra that are relevant for our purposes. In particular, we show that over fields of odd characteristic, the reduced Steenrod algebra acts on the Hochschild (bi)complex corresponding to any braid-word. The cohomology of the resulting complex obtained by having taken Hochschild homology is shown in sections \ref{Bott-Sam} and \ref{Bott-Sam2} to be independent of the presentation (so that it depends, up to isomorphism, only on the braid element underlying the braid-word). We establish these results by working in the category of spectra and taking advantage of properties of Bott-Samelson resolutions. In doing so, we also show that standard arguments by Rouquier on the category of Soergel bimodules have a very natural interpretation in the language of Bott-Samelson resolutions. We expound on this view in the Appendix (section \ref{Appendix}). Specializing to the case $G = PU(r+1)$, our results from section \ref{Bott-Sam2} show that the triply-graded link homology of Khovanov-Rozansky can be defined over any ring where $2$ is invertible (up to a degree shift), and that it supports an action of the reduced Steenrod algebra over the field $\F_p$ for any odd prime $p$. 

\smallskip
\subsection{Assumptions for the Paper}

\smallskip
\noindent

\medskip
\noindent
Throughout this paper, we will assume that $G$ is a connected compact Lie group of adjoint type, and rank $r$ with a fixed root datum.\footnote{In fact, all results of this paper hold for $G$ being the adjoint unitary form of an arbitrary Kac-Moody group.} Even though we are eventually interested in working over fields $\F_p$ for odd $p$, results of individual section of this paper may hold in greater generality. This is indicated at the beginning of each individual section.

\bigskip
\section{Topological enrichment of Soergel modules} \label{one}

\medskip
\noindent
All results of this section are to be understood away from the prime two. In particular, the coefficients in cohomology will be $\Z[\frac{1}{2}]$ (the ring $\Z$ with the prime $2$ inverted) throughout this section unless indicated otherwise. Note that the cohomology of the classifying space $H^\ast(BG_i)$ is well known to be free over the $\Z[\frac{1}{2}]$ for any $G$ and $i$. 

\bigskip
\noindent
Let $T \subset G$ be the maximal torus. Since $G$ is of adjoint type, we use the simple roots to identify $T$ with a rank $r$-torus $(S^1)^{\times r}$. Let $EG$ denote the free contractible $G$-complex that serves as the universal principal $G$-bundle. Notice that the space $EG$ serves as a model of $ET$ and also of $EG_i$. We define the space $\mathcal{F}(\sigma_i)$ as the top left corner in the pullback diagram below, where the map $\rho$ is induced by the right action of $G_i$ on $EG$, and the map $\pi$ is induced by the collapse $G_i/T \longrightarrow pt$ : 
\[
\xymatrix{
EG \times_{T} (G_i/T)   \ar[d]^{\pi} \ar[r]^{\quad \rho} &  EG/T \ar[d]^{j} \\
EG/T      \ar[r]^{j} & EG/G_i.
}
\]

\medskip
\begin{claim} \label{main}
We have the isomorphism of bimodules induced by the maps $\pi^\ast$ and $\rho^\ast$ defined above
\[ H^*(\mathcal{F}(\sigma_i)) = H^\ast(BT) \otimes_{H^*(BG_i)} H^\ast(BT). \]
Furthermore, $H^\ast(\mathcal{F}(\sigma_i))$ is free as a left or right module over $H^*(BT)$ along the maps $\pi^\ast$ and $\rho^\ast$ respectively.  
\end{claim}
\begin{proof}
The calculation of the cohomology of $\mathcal{F}(\sigma_i)$ is a consequence of the Eilenberg-Moore spectral sequence \cite{S} that collapses to the above tensor product due to the fact that $H^\ast(BT)$ is a free module over $H^\ast(BG_i)$. The freeness of $H^\ast(\mathcal{F}(\sigma_i))$ can be verified directly, or by using the fact that the Serre spectral sequence of the fibration given by the map $\pi$ collapses (since both the fiber and base have evenly graded cohomology). 
\end{proof}

\smallskip
\noindent
Now, given a sequence $J_s := (j_1, \ldots, j_s)$, let us define $\mathcal{F}(\tilde{\sigma}_{J_s}) = EG \times_T (G_{j_1} \times_T \cdots \times_T G_{j_s}/T)$. Then we have a family of pullback diagrams:
\[
\xymatrix{
\mathcal{F}(\tilde{\sigma}_{J_s})   \ar[d]^{\pi_s} \ar[r]^{ \rho_{s-1}} &  \mathcal{F}(\sigma_{j_s}) \ar[d]^{\pi} \\
\mathcal{F}(\tilde{\sigma}_{J_{s-1}})      \ar[r]^{ \quad \rho_{s-1}} & EG/T,
}
\]
where $\rho_{s-1}$ is induced by the multiplication map on the first $(s-1)$ factors:
\[ \rho_{s-1} : G_{j_1} \times_T \cdots \times_T G_{j_s}/T \longrightarrow G \times_T G_{j_s}/T. \]
By an induction argument using the above pullback, we easily see that:

\bigskip
\begin{corr} \label{free}
The bimodule $H^\ast(\mathcal{F}(\tilde{\sigma}_{J_s}))$ is a free (left or right) $H^\ast(BT)$-module given by: 
\[ H^\ast(\mathcal{F}(\tilde{\sigma}_{J_s})) = H^\ast(\mathcal{F}(\sigma_{j_1})) \otimes_{H^\ast(BT)} H^\ast(\mathcal{F}(\sigma_{j_2})) \otimes_{H^\ast(BT)} \cdots \otimes_{H^\ast(BT)} H^\ast(\mathcal{F}(\sigma_{j_s})).\]  
Notice that freeness also implies that for any $\Z[\frac{1}{2}]$-algebra $\F$, we have
\[ H^\ast(\mathcal{F}(\tilde{\sigma}_{J_s})) \otimes \F = H^\ast(\mathcal{F}(\tilde{\sigma}_{J_s}), \F) = H^\ast(\mathcal{F}(\sigma_{j_1}), \F) \otimes_{H^\ast(BT, \F)} \cdots \otimes_{H^\ast(BT, \F)} H^\ast(\mathcal{F}(\sigma_{j_s}), \F).\]  

\end{corr}
\noindent
Now observe that the fibration $\pi$ supports two canonical sections $s(0)$ and $s(\infty)$ corresponding to the $T$-fixed points of $G_i/T$ given by $[T]$ and $[\sigma_iT]$ respectively where $[\sigma_iT]$ denotes the coset representing the simple Weyl reflection $\sigma_i \in N_{G_i}(T)/T$. The pullback of the map $\rho$ along $s(0)$ is the identity map on $BT$. On the other hand, the pullback of the map $\rho$ along $s(\infty)$ is the aumorphisms of $BT$ induced by $\sigma_i$. Now if we collapse the section $s(\infty)$ to a point, we obtain the Thom space $BT^{\alpha_i}$ for the line bundle $L(\alpha_i)$ over $BT$ corresponding to the positive root $\alpha_i$. 

\bigskip
\begin{claim} \label{maps}
Consider two topological maps given respectively by taking the topological quotient by the section $s(\infty)$,  and by including the section $s(0)$:
\[ qs(\infty) : \mathcal{F}(\sigma_i) \longrightarrow \mathcal{F}(\sigma_i)/s(\infty) := BT^{\alpha_i}, \quad \quad s(0) : BT \longrightarrow \mathcal{F}(\sigma_i). \]
Then, the induced maps in cohomology: 
\[ rb_i : \tilde{H}^\ast(BT^{\alpha_i}) \longrightarrow H^\ast(\mathcal{F}(\sigma_i)), \quad \quad br_i : H^\ast(\mathcal{F}(\sigma_i)) \longrightarrow  H^\ast(BT) \]
are maps of bimodules over $H^\ast(BT)$, where $\tilde{H}^\ast(BT^{\alpha_i})$ denotes the reduced cohomology of $BT^{\alpha_i}$. 
\end{claim}
\begin{proof}
The claim is clear for the map $br_i$ since the maps $\rho$ and $\pi$ agree with the identity map on the image of $s(0)$. Similarly, one sees that the left $H^*(BT)$-module structure on $H^\ast(\mathcal{F}(\sigma_i))$ agrees with the standard $H^\ast(BT)$-module structure $\tilde{H}^\ast(BT^{\alpha_i})$ along the map $rb_i$. The only thing that is not obvious is the right $H^\ast(BT)$ module structure on $H^\ast(\mathcal{F}(\sigma_i))$ agrees with the standard $H^\ast(BT)$-module structure $\tilde{H}^\ast(BT^{\alpha_i})$ along the map $rb_i$. Now, by restricting along the sections $s(0)$ and $s(\infty)$ one can easily verify that the image of the Thom class $\mbox{th}(\alpha_i) \in \tilde{H}^2(BT^{\alpha_i})$ under $rb_i$ is given by $rb_i(\mbox{th}(\alpha_i)) = \frac{1}{2}(\rho^\ast(\alpha_i) + \pi^\ast(\alpha_i))$. Here we have identified a character with its first Chern class. This class is in fact an integral class (as it restricts to an integral class on the base and fiber of $\pi$). Now by symmetry of this expression of the Thom class, the left and right actions of $H^*(BT)$ on the image of the class $\mbox{th}(\alpha_i)$ agree. This establishes the claim. 
\end{proof}

\smallskip
\noindent
Next, we will need to consider Thom spectra of certain formal, virtual vector bundles. See section \ref{FGL} for details. To describe these objects, let us consider the spherical fibration underlying the vector bundle $\rho^\ast(L(\alpha_i)) \oplus \pi^\ast(L(\alpha_i))$ where $L(\alpha_i)$ is the line bundle corresponding to the character $\alpha_i$. Localizing the sphere away from the prime $2$, one may then uniquely define the formal (one-dimensional) bundle $\zeta_i$ so that the equality $2\zeta_i = \rho^\ast(L(\alpha_i)) \oplus \pi^\ast(L(\alpha_i))$ holds on the level of (localized) spherical fibrations. The Thom spectrum of $-\zeta_i$ is then well-defined away from the prime two as indicated in remark \ref{local}. 

\medskip
\noindent
By theorem \ref{Thomisom}, the cohomology of $\mathcal{F}(\sigma_i)^{-\zeta_i}$ is a rank one free module over the cohomology of $\mathcal{F}(\sigma_i)$, away from the prime two. The definition of the bundle $\zeta_i$ is motivated by the fact that that it restricts to the bundle $L(\alpha_i)$ along $s(0)$. We use this to define: 

\bigskip
\begin{defn} \label{defi}
Let $\mathcal{F}(\sigma_i)^{-\zeta_i}$ denote the Thom spectrum of the (formal virtual) bundle defined by the equality $2\zeta_i = \rho^\ast(L(\alpha_i)) \oplus \pi^\ast(L(\alpha_i))$. As a bimodule, we identify $\tilde{H}^*(\mathcal{F}(\sigma_i)^{-\zeta_i})$ with the bimodule $H^*(\mathcal{F}(\sigma_i))$ under the Thom isomorphism. We may desuspend $rb_i$ with the bundle $\zeta_i$ to obtain a two-term cochain complex 
$F(\sigma_i^{-1}) : H^\ast(BT) \longrightarrow \tilde{H}^\ast(\mathcal{F}(\sigma_i)^{-\zeta_i})$. We also define the cochain complex $F(\sigma_i): H^\ast(\mathcal{F}(\sigma_i)) \longrightarrow  H^\ast(BT)$ to be the map $br_i$. For $\sigma = {\bf{1}}$ the trivial element, we define $F_\bullet({\bf{1}}) = H^\ast(BT)$. By tensoring these elementary complexes together over $H^*(BT)$, we may define cochain complexes $F_\bullet(\tilde{\sigma})$ for arbitrary words. 
\end{defn}

\bigskip
\begin{remark} \label{grading}
The complex $F_\bullet(\tilde{\sigma})$ admits a bi-grading $(\ast, \bullet)$. The grading $\ast$ is called the topological grading that is determined by the cohomology of the underlying spectrum. The grading $\bullet$ is called the (external/cochain) grading that demands that $H^\ast(\mathcal{F}(\sigma_i))$ and $\tilde{H}^\ast(\mathcal{F}(\sigma_i)^{-\zeta_i})$ belong to $\bullet = 0$, as does the entire complex $F_\bullet({\bf{1}})$. 
\end{remark}

\smallskip
\begin{remark} \label{shift}
By construction, the maps in the complex of bimodules $F_\bullet(\tilde{\sigma})$ are induced by maps between the spaces $\mathcal{F}(\tilde{\sigma}_{J_s})$, or Thom spectra of complex vector bundles over them. Notice also that our conventions are compatible with those in \cite{R}, and differ from the ones in \cite{K}. This difference is realized by an overall degree shift in degree in the definition of link homology (see theorem \ref{Linkinv}). 
\end{remark}

\bigskip
\noindent
We now define the Hoschschild bicomplex of the complexes constructed above, which we used to then define Hochschild homology. 

\bigskip
\begin{defn} \label{CHH}
Given the complex of bimodules $F_\bullet(\tilde{\sigma})$, define the Hochschild bicomplex: 
\[ \mbox{CHH}(F_\bullet(\tilde{\sigma}), H^\ast(BT)) = F_\bullet(\tilde{\sigma}) \otimes \Lambda(\epsilon_1, \ldots, \epsilon_r), \quad d(\epsilon_i) = X_i - Y_i, \]
where $\Lambda(\epsilon_1, \ldots, \epsilon_r)$ denotes the exterior algebra over $\Z[\frac{1}{2}]$ in $r$-variables. The variables $X_i, Y_i, \, i=1, \ldots, r$ denotes the left (resp. right) action of the polynomial generators of $H^*(BT)$, and the Koszul differential $d$ is called the (internal) Hochschild differential. 

\medskip
\noindent
As suggested by the notation, $\mbox{CHH}(F_\bullet(\tilde{\sigma}), H^\ast(BT))$ is the derived tensor product of $F_\bullet(\tilde{\sigma})$ and $H^*(BT)$ over the ring $H^*(BT) \otimes H^*(BT)$. Taking homology with respect to the Hochschild differential gives rise to the Hoschschild homology complex $\mbox{HH}(F_\bullet(\tilde{\sigma}), H^\ast(BT))$ with the residual (external) differential induced by the original complex $F_\bullet(\tilde{\sigma})$.\footnote{One of the results of this paper will be to show that this homology is independent of the choice of braid word.}

\medskip
\noindent
Note that if $\F$ is any $\Z[\frac{1}{2}]$-algebra, one can just as well define the Hochschild bicomplex over $\F$
\[ \mbox{CHH}(F_\bullet(\tilde{\sigma}), H^\ast(BT)) \otimes \F = \mbox{CHH}(F_\bullet(\tilde{\sigma}), H^\ast(BT,\F)).  \]

\end{defn}

\medskip
\begin{remark} \label{grading2}
We will give the complex $\mbox{CHH}(F_\bullet(\tilde{\sigma}), H^\ast(BT))$ a tri-grading $(\star, \ast, \bullet)$ by extending the grading of \ref{grading} by an internal grading $\star$, and demanding that the classes $\epsilon_i$ have grading $(-1,2,0)$. Notice that the (internal) Hochschild differential only raises the first or internal grading by one, while the remaining (external) differential raises the third or external grading by one. In other words, these differentials have grading $(1,0,0)$ and $(0,0,1)$ respectively. 

\end{remark}

\section{The reduced mod-$p$ Steenrod Algebra and its action on Soergel bimodules and their Hochschild bicomplex in odd characteristic} \label{FGL}

\medskip
\noindent
Throughout this section, we work over a field of odd characteristic $p$. In particular, all cohomology groups are assumed to have coefficients in $\F_p$ unless indicated otherwise. The mod-$p$ Steenrod algebra is the algebra of stable cohomology operations acting on cohomology with coefficients in $\F_p$. In particular it acts on the terms in the mod-$p$ reduction of the cochain complex $F_\bullet(\tilde{\sigma})$ defined in the previous section (see corollary \ref{free}). 

\medskip
\noindent
Working with odd primes, let us recall that the Steenrod algebra is generated by elements denoted by the Bockstein $\beta$, that increases cohomological degree by one, and reduced Steenrod powers $\mathcal{P}^i$ that increase cohomological degree by $2i(p-1)$, for $i \geq 0$, with $\mathcal{P}^0$ being the identity operator. Since all the terms in the cochain complex $F_\bullet(\tilde{\sigma})$ are evenly graded, we may ignore the Bockstein and only consider the {\em reduced Steenrod algebra} which is defined as the subalgebra generated by the reduced Steenrod powers $\mathcal{P}^i$. 

\medskip
\begin{defn} \label{redSteenrod}
The reduced Steenrod Algebra $\mathcal{A}_p$ is defined as the ring of stable chomology operations generated by the operators $\mathcal{P}^i$ for $i \geq 0$. 
\end{defn}

\medskip
\noindent
These operators are known to satisfy the {\em Adem relations} given by: 
\[ \mathcal{P}^i \mathcal{P}^j = \sum_k (-1)^{i+k} \binom{(p-1)(j-k)-1}{i-pk} \mathcal{P}^{i+j-k}\mathcal{P}^k, \quad \quad \mbox{for} \quad  i < pj.\]
On the cohomology of a topological space (so that the cohomology forms a graded commutative ring), these operators satisfy the following formulas known as the Cartan formula and instability resp.: 
\[  \mathcal{P}^k(x \cup y) = \sum_{i+j=k} \mathcal{P}^i(x) \cup \mathcal{P}^j(y), \quad \quad \mathcal{P}^n(x) = 0, \quad \mbox{if} \quad 2n > |x|. \]
Notice that the Cartan formula has an elegant formulation in terms of the {\em total} Steenrod power operation $\mathcal{P} := \sum_i \mathcal{P}^i$: 
\[ \mathcal{P}(x \cup y) = \mathcal{P}(x) \cup \mathcal{P}(y), \quad \quad  \mbox{where} \quad \mathcal{P}(z) := \sum_i \mathcal{P}^i(z). \]

\medskip
\begin{example}
An important example is given by the action of $\mathcal{A}_p$ on the cohomology ring of $\C P^{\infty}$, $H^*(\C P^{\infty}, \F_p) = \F_p[x]$. This action is uniquely determined by the Cartan formula and the equality
\[ \mathcal{P}(x) = x + x^p. \]
\end{example}

\medskip
\noindent
There is a well known generalization of the Cartan formula that applies to the cohomology of Thom spectra, which we will need presently. Let us very briefly discuss the Thom spectrum as an object in the category of spectra, see \cite{ABGHR} for more details. The category of spectra is essentially the localization of the category of CW complexes where the suspension functor is invertible. In particular, any homology theory defined on the category of CW complexes factors through spectra. 

\medskip
\noindent
A Thom spectrum is the only kind of spectrum we will require for the purposes of this document. To define a Thom spectrum for a virtual vector bundle $E$ of dimension $n$ over a topological space $B$, one first classfies the stable spherical fibration $S(E)$ underlying $E$ via a map:
\[ S(E) : B \longrightarrow BGl(S), \]
where $Gl(S)$ is a grouplike monoid of stable automorphism of the sphere. The classifying space of $Gl(S)$ classfies stable spherical fibrations in much the same way as the classifying space $BO$ of the stable orthogonal group $O := O(\infty)$ classifies stable vector bundles. The group $Gl(S)$ acts on the sphere spectrum, which is the topological space $S^0$ seen as a spectrum. By (de)suspending suitably, we obtain an action of $Gl(S)$ on the $n$-sphere spectrum $S^n$. The classifying map $S(E)$ then gives rise to a parametrized spectra over the space $B$ with any fiber being equivalent to the $n$-sphere spectrum. By definition, this parametrized spectrum admits a basepoint-section. Pinching this section off gives rise to a natural object in the category of spectra known as the Thom spectrum denoted by $B^E$. Notice that if we started with an honest vector bundle $E$, this construction recovers the familiar topological space given by pinching off the sphere bundle of $E$ from the disc bundle. 

\medskip
\noindent
The Thom isomorphism is a classical theorem that describes the cohomology of the Thom spectrum $B^E$ as a rank one free module over the cohomology of $B$. 

\medskip
\begin{thm} \label{Thomisom} \cite{A}
Let $E$ denote an $n$-dimensional oriented vector bundle $E$ over a space $B$, and let $\F$ denote an arbitrary coefficient ring. Then there exists a unique integral cohomology class (called the Thom class) $\mbox{th}(E) \in \tilde{H}^n(B^E, \F)$, such that multiplication with $\mbox{th}(E)$ exhibits $\tilde{H}^\ast(B^E, \F)$ as a free rank one module over $H^\ast(B, \F)$. 
\end{thm}

\medskip
\begin{remark}
If $\F$ is taken to be the field $\F_p$ of odd characteristic in theorem \ref{Thomisom}, then the action of $\mathcal{A}_p$ on $\tilde{H}^*(B^E, \F_p)$ is determined by the action on $H^*(B,\F_p)$ and on the Thom class $\mbox{th}(E)$, via a version of the Cartan formula
\[  \mathcal{P}(\mbox{th}(E) \cup y) = \mathcal{P}(\mbox{th}(E)) \cup \mathcal{P}(y), \quad \quad y \in H^*(B, \F_p). \]
\end{remark}

\medskip
\begin{example} \label{Thomsum}
Let $\C P^{\gamma}$ denote the Thom space of the tautological line bundle over $\C P^{\infty}$. Invoking naturality of the action of $\mathcal{A}_p$ on the pinch map $\C P^{\infty} \rightarrow \C P^{\gamma}$, we see that $\mathcal{A}_p$ acts on the Thom class $\mbox{th}(\gamma) \in \tilde{H}^2(\C P^{\gamma}, \F_p)$ 
\[ \mathcal{P}(\mbox{th}(\gamma)) = \mbox{th}(\gamma) \cup (1 + x^{p-1}). \]
This formula readily generalizes to the Thom space ${(\C P \times \C P)}^{\gamma_1 \oplus \gamma_2}$ of the sum of line bundles $\gamma_1 \oplus \gamma_2$ over $\C P^{\infty} \times \C P^{\infty}$
\[ \mathcal{P}(\mbox{th}(\gamma_1 \oplus \gamma_2)) = \mbox{th}(\gamma_1 \oplus \gamma_2) \cup (1 + x_1^{p-1}) (1+x_2^{p-1}). \]
\end{example}

\medskip
\begin{remark} \label{local}
In the construction of the Thom spectrum $B^E$ given above, we can work locally away from a prime $q$. In other words, one may start with a formal virtual spherical fibration, which is defined by map of the form 
\[ S(E) : B \longrightarrow BGl(S[1/q]), \]
where the grouplike monoid $Gl(S[1/q])$ acts on the localization of sphere spectrum $S^0[1/q]$, away from the prime $q$. The construction of the Thom spectrum goes through as before to give rise to a spectrum $B^E$ that is well behaved away from $q$. All the results of this section that apply to Thom spectra of virtual vector bundles, also apply to Thom spectra of formal virtual spherical fibrations away from a fixed prime. In particular, the Thom isomorphism theorem \ref{Thomisom} holds for such formal Thom spectra if $\F$ is a ring in which $q$ is invertible. The case $q=2$ is relevant for our purposes below. 
\end{remark}

\smallskip
\begin{remark}
Let $H^*(X, \F_p)$ and $H^*(Y, \F_p)$ be the cohomology groups of spectra $X$ and $Y$ resp. with their respective actions of $\mathcal{A}_p$. Then the Cartan formula can be used to endow the tensor product $H^*(X, \F_p) \otimes_{\F_p} H^*(Y, \F_p)$ with an action of $\mathcal{A}_p$. This action can be identified with the canonical action of $\mathcal{A}_p$ on $H^*(X \wedge Y, \F_p)$ via the Kunneth formula that identifies $H^*(X \wedge Y, \F_p)$ with $H^*(X, \F_p) \otimes_{\F_p} H^*(Y, \F_p)$. \end{remark}

\medskip
\noindent
The above properties of $\mathcal{A}_p$ can be used to describe an action on Soergel bimodules. So for instance, consider the cohomology groups $\tilde{H}^\ast(\mathcal{F}(\sigma_i)^{-\zeta_i}, \F_p)$, with Thom class $\mbox{th}(-\zeta_i)$, where we recall that $\zeta_i$ was the formal line bundle defined to satisfy an equality on the level of spherical fibrations $2\zeta_i = \rho^\ast(L(\alpha_i)) \oplus \pi^\ast(L(\alpha_i))$. By Thom isomorphism, $\tilde{H}^\ast(\mathcal{F}(\sigma_i)^{-\zeta_i}, \F_p)$ is a rank one free module over $H^\ast(\mathcal{F}(\sigma_i), \F_p)$ generated by $\mbox{th}(-\zeta_i)$. Invoking the Cartan formula, the action of $\mathcal{A}_p$ on $\tilde{H}^\ast(\mathcal{F}(\sigma_i)^{-\zeta_i}, \F_p)$ is determined by its action on $H^\ast(\mathcal{F}(\sigma_i), \F_p)$, and on the class $\mbox{th}(-\zeta_i)$. By claim \ref{main}, $H^\ast(\mathcal{F}(\sigma_i), \F_p)$ is a quotient of the cohomology of a product of infinite projective spaces $\C P^{\infty}$. Consequently, we can completely describe the action of $\mathcal{A}_p$ on $H^\ast(\mathcal{F}(\sigma_i), \F_p)$. It remains to describe the action of $\mathcal{A}_p$ on $\mbox{th}(-\zeta_i)$. We begin by invoking naturality and the example \ref{Thomsum} on $\rho^\ast(L(\alpha_i)) \oplus \pi^\ast(L(\alpha_i))$
\[ \mathcal{P}(\mbox{th}(2\zeta_i)) = \mbox{th}(2\zeta_i) \cup (1+ \rho^\ast(\alpha_i)^{p-1}) (1 + \pi^\ast(\alpha_i)^{p-1}), \]
where we have identified the character $\alpha_i$ with its first Chern class $\alpha_i \in H^2(BT, \F_p)$. We may then formally take the $-\frac{1}{2}$ power of $\zeta_i$ (recall that we are working over odd primes):

\bigskip
\begin{thm} \label{Thom}
The action of $\mathcal{A}_p$ on $\mbox{th}(-\zeta_i) \in \tilde{H}^{-2}(\mathcal{F}(\sigma_i)^{-\zeta_i}, \F_p)$ is defined formally by
\[ \mathcal{P}(\mbox{th}(-\zeta_i)) = \mbox{th}(-\zeta_i) \cup (1+ \rho^\ast(\alpha_i)^{p-1})^{-\frac{1}{2}} (1 + \pi^\ast(\alpha_i)^{p-1})^{-\frac{1}{2}}, \]
where, as before, we have identified the character $\alpha_i$ with its first Chern class $\alpha_i \in H^2(BT, \F_p)$ and $\rho$ and $\pi$ are the maps defined in claim \ref{main}. Now, using corollary \ref{free} and by invoking the Cartan formula and naturality of the $\mathcal{A}_p$-action with respect to maps of spectra, we obtain an explicit description of the action of $\mathcal{A}_p$ on the mod-$p$ reduction of the cochain complexes $F_\bullet(\tilde{\sigma})$ of definition \ref{defi} for arbitrary braid words $\tilde{\sigma}$. 
\end{thm}

\bigskip
\noindent
We now turn our attention to the mod-$p$ reduction of the Hochschild bicomplex of $F_\bullet(\tilde{\sigma})$. Recall from \ref{CHH} that this bicomplex is given by: 
\[ \mbox{CHH}(F_\bullet(\tilde{\sigma}), H^\ast(BT, \F_p)) = F_\bullet(\tilde{\sigma}) \otimes \Lambda(\epsilon_1, \ldots, \epsilon_r) \otimes \F_p, \quad d(\epsilon_i) = X_i - Y_i, \]
where $X_i, Y_i, \, i=1, \ldots, r$ denotes the left (resp. right) action of $H^*(BT, \F_p) = \F_p[X_1, \cdots, X_r]$ and $d$ is called the (internal) Hochschild differential. 

\medskip
\noindent
To extend the $\mathcal{A}_p$-action to $\mbox{CHH}(F_\bullet(\tilde{\sigma}), H^\ast(BT, \F_p))$ we simply need to {\em define} the $\mathcal{A}_p$ action on the classes $\epsilon_i$ that commutes with the differential. Extending this to the entire bicomplex is simply a matter of applying the Cartan formula. Consequently, we have:

\bigskip
\begin{thm} \label{twist}
The mod-$p$ reduction of the Hochschild bicomplex $\mbox{CHH}(F_\bullet(\tilde{\sigma}), H^\ast(BT, \F_p))$ admits an action of $\mathcal{A}_p$ defined on the generators $\epsilon_i$ by
\[ \mathcal{P}(\epsilon_i) = \epsilon_i (1 + (X_i-Y_i)^{p-1}). \]
Furthermore, this action extends the action on $F_\bullet(\tilde{\sigma}) \otimes \F_p$ given in theorem \ref{Thom}, via the Cartan formula. Note that the action of $\mathcal{A}_p$ only impacts the topological grading in the expected fashion. 
\end{thm}

\begin{proof}
It is straightforward to see that this definition commutes with the differential. So it only remains to show that the above definition satisfies the Adem relations. This is straightforward to verify by noticing that this action is abstractly isomorphic to the action of $\mathcal{A}_p$ on the Thom class of the bundle $\gamma_1 \otimes \gamma_2^{\ast}$ on $\C P^{\infty} \times \C P^{\infty}$. 
\end{proof}

\bigskip
\section{Independence of presentation: The Braid relation and Bott-Samelson resolutions} \label{Bott-Sam}

\medskip
\noindent
So far we have worked with a fixed presentation of a braid word $\tilde{\sigma} = \sigma_{i_1}^{\epsilon_1} \, \sigma_{i_2}^{\epsilon_2} \cdots \sigma_{i_k}^{\epsilon_k}$. One would like to know how much of the structure developed in the previous sections remains preserved if we change the presentation. In \cite{KR} the authors show that in characteristic zero, the twisted action of the Witt algebra on the homology of the complex $HH_*(F_\bullet(\tilde{\sigma}), H^\ast(BT, \Q))$ is independent of presentation (it is in fact, a link invariant!). 

\medskip
\noindent
In this section, we will 
work over the integers $\Z$ and so all cohomology is assumed to be with coefficients in $\Z$ unless otherwise indicated (like for instance in theorem \ref{quasi}).

\bigskip
\noindent
Let $i, j$ be two elements in the set indexing the simple reflections in the Weyl group. Let us adopt the convention from \cite{R} of denoting a general element of the set $\{i, j\}$ by the letter $\nu$, and the complement element by the letter $-\nu$. Consider a braid relation of the form:
\[ \sigma := \cdots \sigma_{\nu} \, \sigma_{-\nu} \, \sigma_\nu = \cdots \sigma_{-\nu} \, \sigma_\nu \, \sigma_{-\nu} \quad \quad (\mbox{ $m_{ij}$ factors}), \]
For suitable integers $m_{i,j}$ depending on the corresponding entries of the Cartan Matrix and with $n \leq m_{ij} < \infty$. We will denote by $\tilde{\sigma}(\nu,n)$ the abstract braid word:
\[ \tilde{\sigma}(\nu,n) = \cdots \sigma_\nu \, \sigma_{-\nu} \, \sigma_\nu  \quad \quad (\mbox{$n$ factors}). \]

\medskip
\noindent
Let us now recall the complex $F_\bullet(\nu,n)$ of bimodules corresponding to the word $\tilde{\sigma}(\nu,n)$ as described in \ref{defi}. Let $F(\sigma_i^{-1}): H^\ast(BT) \longrightarrow \tilde{H}^\ast(\mathcal{F}(\sigma_i)^{-\zeta_i}) = H^{\ast+2}(\mathcal{F}(\sigma_i))$ graded so that $H^\ast(BT)$ is in degree $-1$. Also let $F(\sigma_i) : H^\ast(\mathcal{F}(\sigma_i)) \longrightarrow  H^\ast(BT)$ to be the map $br_i$,  graded so that $H^*(BT)$ is in degree $+1$. By tensoring these elementary complexes together as prescribed by the word $\tilde{\sigma}(\nu,n)$, we may define complex of bimodules $F_\bullet(\nu,n)$. 

\bigskip
\noindent
 In this section, we will use topological arguments to establish a canonical zig-zag of equivalences between $F_\bullet(\nu, m_{ij})$ and $F_\bullet(-\nu, m_{ij})$. Consequently, we will establish a topologically induced equivalence between the complexes $F_\bullet(\tilde{\sigma})$ corresponding to any two presentations of the same braid group element (see theorem \ref{quasi} for a precise statement).

\bigskip
\noindent
We begin by invoking some standard facts about the topology of flag varieties for Kac-Moody groups. A standard reference for these results is \cite{Ku}. Let $P \subset G$ denote the rank two parabolic corresponding to the reflections $\sigma_\nu$ and $\sigma_{-\nu}$. Let $G_{\pm \nu}$ denote the rank one parabolics corresponding to the individual reflections $\sigma_{\pm \nu}$. By $Z(\nu,n) \subseteq P/T$ we shall mean the Schubert variety given by the closure of all Schubert cells corresponding to elements $w$ in the Weyl group such that $w$ has a presentation given by a sub-word of  $\tilde{\sigma}(\nu,n)$. The space $Z(\nu,n)$ is known to be a CW complex of dimension $2n$, with $T$-invariant cells in even degree. For $k < n$, the $2k$-skeleton of $Z(\nu,n)$ is given by the union of subcomplexes $Z(\nu,k) \cup Z(-\nu,k)$. We will use the notation $\mathcal{X}(\nu,n)$ to denote the spaces $EG \times_T Z(\nu,n)$. 

\medskip
\noindent
The left $H^\ast(BT)$-module $H^\ast(EG \times_T (P/T))$ supports a Schubert basis denoted by $\delta_{\nu,n}$ indexed by elements $\tilde{\sigma}(\nu,n)$ for $n \leq m_{ij}$. We set $\delta_{\nu,1} = \delta_\nu$. By definition, the Schubert basis pairs diagonally with the generators of $H_{2k}(\mathcal{X}(\nu,k))$ (see \cite[Ch. 9, \S 11]{Ku}). We will denote the restrictions of the Schubert basis to sub-skeleta by the same name. 

\bigskip
\begin{defn}
Define a cochain complex of bimodules $X_\bullet(\nu,n)$ with non-trivial terms given by:
\[ X_{k}(\nu,n) := H^\ast(\mathcal{X}(\nu,n-k)) \oplus H^\ast(\mathcal{X}(-\nu,n-k)), \quad  \mbox{if $0 < k < n$}, \]
\[ X_{0}(\nu,n) := H^\ast(\mathcal{X}(\nu,n)), \quad \quad X_n(\nu,n) := H^\ast(BT). \]
The differentials $d_k : X_{k}(\nu, n) \longrightarrow X_{k+1}(\nu,n)$ are given by restrictions with appropriate signs as:
\[  d_k (\alpha \oplus \beta) = (\alpha+(-1)^{n-k-1}\beta) \, \oplus \, (\beta + (-1)^{n-k-1}\alpha), \quad \mbox{if $0 < k < (n-1)$}. \]
The differentials $d_0$ and $d_{n-1}$ are again defined by cellular restrictions with appropriate signs:
\[   d_0 (\alpha) = (\alpha\, \oplus \, (-1)^{n-1}\alpha), \quad  \quad d_{n-1} (\alpha \oplus \beta) = \alpha+\beta. \]

\bigskip
\begin{claim} \label{acyclic}
The homology of the complex $X_\bullet(\nu,n)$ is a free rank one left (or right) $H^\ast(BT)$-module in external degree (or degree represented by $\bullet$) $0$, generated by the Schubert cell $\delta_{\nu,n}$. 
\end{claim}
\begin{proof}
Let us decompose the complex $(X_\bullet(\nu,n),\,  d_\bullet)$ into a sum of complexes $(X^j_\bullet(\nu,n), \, d^j_\bullet)$ indexed by the Schubert basis, for $0 \leq j \leq n$. In other words, we define: 
\[ X^j_{k}(\nu,n) = \langle \, H^*(BT) \, \delta_{\pm \nu,j} \, \rangle \subseteq X_{k}(\nu,n), \quad j < n. \]
For $j = n$, we define the complex $X^n_\bullet(\nu,n)$ to take the value $H^*(BT) \, \delta_{\nu, n}$ for $\bullet = 0$ and zero for all other external degrees. Note by definition that $X^j_{k}(\nu,n) = 0$ if $n-k < j$. We now proceed to show that all these complexes except for the one corresponding to $j=n$ are acyclic. Now recall that the spaces $Z(\pm \nu, n-k)$ are sub-skeleta of $Z(\nu, n)$ for $0 < k$. It follows that the Schubert basis $\delta_{\pm \nu, j} \in X_{k}(\nu,n)$ lift to unique Schubert basis elements in $H^{2j}(\mathcal{X}(\nu,n))$ for $j \leq n-k < n$. The formula for $d_k$ given above now shows that for $0 < j \leq n-k < n$ we have:
\[ \mbox{Kernel} \, d^j_{k} = \mbox{Image} \, d^j_{k-1} = \{ (z, (-1)^{n-k} z) \in X_{k}^j(\nu,n) \, \,  |  \, \,  z \in H^\ast(\mathcal{X}(\nu, n))\}. \]
In other words, the complex $X^j_\bullet(\nu,n)$ is acyclic for all values of $j$ except for $j=0$ and $j=n$. Acyclicity for the case $j=0$ it is straightforward to check by hand. For $j=n$, the complex has a single generator represented by the Schubert cell $\delta_{\nu,n}$, which belongs to the external degree $0$. The result follows. 
\end{proof}

\medskip
\begin{remark} \label{homology}
Notice also that by definition of the spaces $\mathcal{X}(\nu,n)$, we have a homeomorphism between the spaces $\mathcal{X}(\nu,m_{ij})$ and $\mathcal{X}(-\nu,m_{ij})$. This implies that there is a canonical isomorphism: 
\[ X_\bullet(\nu,m_{ij}) = X_\bullet(-\nu,m_{ij}). \]
\end{remark}
\end{defn}

\bigskip
\noindent
Consider the map induced by multiplication in $P$, known as the Bott-Samelson resolution:
\[ \mbox{BS}(\nu,n) : \mathcal{F}(\nu, n) \longrightarrow \mathcal{X}(\nu, n). \]

\bigskip
\begin{claim}
The map $\mbox{BS}(\nu,n)$ induces an injective map of complexes of bimodules:
\[  \mbox{BS}_\bullet(\nu,n) : X_{\bullet}(\nu, n) \longrightarrow F_\bullet(\nu,n). \]
\end{claim}
\begin{proof}
We shall construct the required map $\mbox{BS}_\bullet(\nu,n)$ by induction, starting with a map:
\[ \mu_\bullet : X_\bullet(\nu, n) \longrightarrow X_\bullet(-\nu, n-1) \otimes_{H^\ast(BT))} F(\sigma_\nu). \]
Topologically, the map $\mu_\bullet$ will be realized inside a pair of pullback diagrams:
\[
\xymatrix{
   EG \times_T (\tilde{Z}(-\nu,n-1) \times_T G_\nu/T)  \ar[d]^{\pi} \ar[r]^{\quad \quad \quad \quad \mu} & \mathcal{X}(\nu,n) \ar[d]^{\pi_\nu} \ar[r]^{\quad \rho} & EG/T \ar[d]^{j} \\
 \mathcal{X}(-\nu,n-1) \ar[r]^{\kappa} & EG \times_T (\tilde{Y}(\nu,n)/G_\nu) \ar[r] & EG/G_\nu 
}
\]
where $\mu$ is induced by group multiplication in $P$, and factors the Bott-Samelson resolution. The space $\tilde{Z}(\nu,n)$ is the subspace of $P$ given by the pre image of $Z(\nu,n) \subseteq P/T$. The space $\tilde{Y}(\nu,n)$ is similarly defined by as the pre image of $Y(\nu,n) \subseteq P/G_\nu$, with $Y(\nu,n)$ being the Shubert variety in $P/G_\nu$ given by the image of $ of Z(\nu,n)$. The map $\kappa$ above is given by the restriction of $\pi_\nu$ to the subspace $\mathcal{X}(-\nu,n-1) \subset \mathcal{X}(\nu,n)$. 

\smallskip
\noindent
Let $  \mathcal{X}(\pm \nu,n-1,\nu)$ denote the space $EG \times_T (\tilde{X}( \pm \nu,n-1) \times_T G_\nu/T)$. Using the Eilenberg-Moore spectral sequence, as in corollary \ref{free}, we see that: 
\[ H^\ast(\mathcal{X}(\pm \nu,n-1,\nu)) = H^\ast(\mathcal{X}( \pm \nu, n-1)) \otimes_{H^\ast(BT))} H^\ast(\mathcal{F}(\sigma_\nu)) = H^\ast(\mathcal{X}( \pm \nu, n-1)) \otimes F_{0}(\sigma_\nu). \]
Under the above identification, we extend the map $\mu^\ast$ to $\mu_\bullet$ on the level of complexes:

\smallskip
\noindent
For an element $\gamma \in H^\ast(\mathcal{X}(\nu, n-k)) \subset X_{k}(\nu, n)$, define $\mu_{k}$ by:
\[ \mu_{k}(\gamma) = \mu^\ast(\gamma) \, \in \, H^\ast(\mathcal{X}(-\nu, n-k-1)) \otimes F_{0}(\sigma_\nu). \]
For an element $\beta \in H^\ast(\mathcal{X}(-\nu, n-k)) \subset X_{k}(\nu, n)$, we define:
\[ \mu_{k}(\beta) = - \mu^\ast(\beta) \, \oplus \, \beta \, \in \, H^\ast(\mathcal{X}(\nu, n-k-1)) \otimes F_{0}(\sigma_\nu) \, \oplus \, H^\ast(\mathcal{X}(-\nu,n-k)), \]
where we have also used the letter $\mu$ to denote the (degenerate) Bott-Samelson map: 
\[ \mu : \mathcal{X}(\nu, n-k-1, \nu) \longrightarrow \mathcal{X}(\nu, n-k-1) \subset \mathcal{X}(-\nu,n-k). \]
It is straightforward (albeit tedious) to check that $\mu_\bullet$ commutes with the differential. 
\end{proof}

\bigskip
\begin{thm} \label{quasi}
Let $A = (a_{ij})$ be the Cartan matrix for $G$. If all primes $p \leq a_{ij} a_{ji}$ have been inverted, then $\mbox{BS}_\bullet(\nu, n)$ is a  homotopy equivalence, and consequently, the homotopy type of the complex of bimodules $F_\bullet(\tilde{\sigma})$ is independent of presentation away from such primes. Over the integers, the Bott-Samelson map always induces a quasi-isomorphism:
\[  \mbox{BS}_\bullet(\nu,n) : X_{\bullet}(\nu, n) \longrightarrow F_\bullet(\nu,n). \]
In particular, the quasi-isomorphism type of $F_\bullet(\tilde{\sigma})$ is always independent of presentation regardless of the characteristic. 

\end{thm}
\begin{proof}
Let us first prove that the Bott-Samelson map is always a quasi-isomorphism. Recall from \ref{acyclic} that the homology of $X_\bullet(-\nu,n-1)$ is concentrated in degree $0$, and is generated by the Schubert class $\delta_{-\nu,n-1}$. Each step in the inductive construction of $\mbox{BS}_\bullet(\nu,n)$ is easily seen to preserve the homology generator, and is therefore a quasi-isomorphism. This inductive construction also shows that the cokernel of $\mbox{BS}_\bullet(\nu,n)$ is a complex of bimodules that has trivial homology, and is free as a complex of left (or right) $H^\ast(BT)$-modules. Hence, tensoring the map $\mbox{BS}_\bullet(\nu,n)$ on the left by a complex of the form $F_\bullet(\tilde{\sigma}')$ and on the right by another complex of the form $F_\bullet(\tilde{\sigma}'')$ preserves the fact that it is a quasi-isomorphism. Consequently, replacing $F_\bullet(\nu, m_{ij})$ with $F_\bullet(-\nu, m_{ij})$ in a presentation for $F_\bullet(\tilde{\sigma})$ preserves the quasi-isomorphism type. 

\smallskip
\noindent
Now assume that all primes $p \leq a_{ij}a_{ji}$ have been inverted, then we will show in the appendix that one may split off trivial summands from $F_\bullet(\nu, n)$ in the same was as done in \cite[proof of Prop. 9.2]{R}. In particular, one obtains a retraction from $F_\bullet(\nu, n)$ onto the image of $\mbox{BS}_\bullet(\nu, n)$. This establishes the homotopy equivalence of the Bott-Samelson map away from the primes $p \leq a_{ij}a_{ji}$, where $A=(a_{ij})$ is the Cartan matrix of $G$. 
\end{proof}

\medskip
\begin{remark} \label{quasi-p}
Notice that all maps that establish the quasi-isomorphism in the statement of the above theorem are induced by topological maps. Consequently, working mod any prime $p > a_{ij} a_{ji}$, the quasi-isomorphism of theorem \ref{quasi} respects the action of the reduced Steenrod algebra $\mathcal{A}_p$. 
\end{remark}

\section{Independence of presentation: The identity $F(\sigma_i) \otimes_R F(\sigma_i^{-1}) \simeq R$ and triply graded link homology away from the prime $2$} \label{Bott-Sam2}

\bigskip
\noindent
In this section we work over the ring $\Z[\frac{1}{2}]$, so all cohomology is assumed with coefficients in $\Z[\frac{1}{2}]$ unless otherwise indicated. 
Given an index $i$, recall that $F(\sigma_i) \otimes_R F(\sigma_i^{-1})$ is represented by the following complex nontrivial in degrees $\bullet = \{-1,0,1 \}$, with maps induced by the generating maps \ref{defi}:
\[ \mathcal{C} : H^\ast(\mathcal{F}(\sigma_i)) \longrightarrow H^\ast(BT) \, \, \bigoplus \, \, H^\ast(\mathcal{F}(\sigma_i)) \otimes_{H^\ast(BT)} \tilde{H}^\ast(\mathcal{F}(\sigma_i)^{-\zeta_i}) \longrightarrow \tilde{H}^\ast(\mathcal{F}(\sigma_i)^{-\zeta_i}). \]
Now Rouquier shows in \cite{R} that this complex is (abstractly) homotopy equivalent to the complex of bimodules with a single non-zero term $H^\ast(BT)$ placed in degree zero. In particular, the complex above is homotopy equivalent to its homology as a bimodule. This homology is straightforward to compute:
\[ H^*(\mathcal{C}) = \frac{\delta_i \otimes \tilde{H}^\ast(\mathcal{F}(\sigma_i)^{-\zeta_i})}{\langle \delta_i \otimes \mbox{th}(-\zeta_i) \cup rb_i(\mbox{th}(\alpha_i)) \rangle} \subset \frac{H^\ast(\mathcal{F}(\sigma_i)) \otimes \tilde{H}^\ast(\mathcal{F}(\sigma_i)^{-\zeta_i})}{\mbox{Image} \, H^\ast(\mathcal{F}(\sigma_i))} , \]
where $rb_i(\mbox{th}(\alpha_i)) \in H^\ast(\mathcal{F}(\sigma_i))$ is the image in cohomology of the Thom class of the bundle obtained by collapsing the section $s(\infty)$ in $\mathcal{F}(\sigma_i)$ (see claim \ref{maps}). Similarly, $\delta_i$ is the Schubert basis element corresponding to $\sigma_i$, which can be identified with the image of the Thom class of the bundle obtained by collapsing the section $s(0)$. One may write $\delta_i$ explicitly as $ \frac{1}{2}(\rho^\ast(\alpha_i) - \pi^\ast(\alpha_i))$. The homology $H^\ast(\mathcal{C})$ is a rank one free $H^\ast(BT)$-module with identical left and right actions. The generator is given by the class $\delta_i \otimes \mbox{th}(-\zeta_i)$, which can be easily seen to be invariant under $\mathcal{A}_p$ if one works over the field $\F_p$. 

\bigskip
\begin{remark}
A similar argument to the one given above establishes the opposite quasi-isomorphism $F(\sigma_i^{-1}) \otimes_R F(\sigma_i) = R$. We leave it to the reader to formulate it. In addition, reducing this quasi-isomorphism mod any odd prime, it is easy to see that the resulting quasi-isomorphism respects the action of the reduced Steenrod algebra. 
\end{remark}

\subsection{Triply graded link homology}

\medskip
\noindent

\medskip
\noindent
Let us now consider the groups $G = PU(r+1)$, for $r \geq 1$. We fix a $\Z[\frac{1}{2}]$-algebra $\F$. Given a link $L$, let $\sigma$ denote a braid representing the link $L$ so that $L$ can be obtained by glueing the ends of $\sigma$. Let us pick a presentation $\tilde{\sigma}$ of $\sigma$ and consider the complex $HH_*(F_\bullet(\tilde{\sigma}), H^*(BT, \F))$ as defined in definition \ref{CHH}. Recall that this complex has a tri-grading $(\star, \ast, \bullet)$ as defined in remark \ref{grading2}. We define the triply-graded link homology of the presentation $\tilde{\sigma}$ over $\F$ as: 
 
 \medskip
 \begin{defn} \label{TGLH}
 Given a presentation $\tilde{\sigma}$ of the braid (of at least two strands) obtained by unraveling a link $L$, we define the triply-graded link homology of $L$ over an $\Z[\frac{1}{2}]$-algebra $\F$  as:
 \[ H_*(L) := H_*(HH_*(F_\bullet(\tilde{\sigma}), H^*(BT,\F)). \]
We define the homology of the unknot with one component to be the underlying ring $\F$ graded in degree $(0,0,0)$. 
 \end{defn}
 
 \smallskip
 \noindent
 We can now generalize the main result of \cite{K}. 
 
 \medskip
 \noindent
 \begin{thm} \label{Linkinv}
 The tri-graded groups $H_\ast(L)$ are well-defined over any ring $\F$ which is an $\Z[\frac{1}{2}]$-algebra and depend only on the link $L$ up to global shifts by degree $(-1,2,1)$. Furthermore, over fields $\F_p$ for odd primes $p$, the link homology of $L$ admits an action of the reduced Steenrod algebra $\mathcal{A}_p$ (see theorem \ref{twist}) that commutes with this global shift. 
 \end{thm}

\begin{proof}
To verify this claim, one needs to verify invariance under the three Reidemeister moves. We have essentially established invariance under the last two moves. For instance, theorem \ref{quasi} shows that $H_*(L)$ is invariant under the third Reidemeister move, and the main result of this section shows invariance under the second. It remains to establish invariance under the first move. In other words, we need to establish that for $G = PU(n+2)$ and maximal torus $T^{n+1}$, we have an isomorphism up to a shift in degree $(-1,2,1)$ that respects the action of the reduced Steenrod algebra $\mathcal{A}_p$ if one works over $\F_p$:
\[ H_*(HH_*(F_\bullet(\sigma_i), H^*(BT^{n+1},\F)) = H_*(HH_*(F_\bullet({\bf{1}}_{n+1}), H^*(BT^n,\F)), \]
where ${\bf{1}}_{n+1}$ denotes the unit element of the $n+1$-stranded Braid group. The above groups can be easily seen to be isomorphic to  $\F[X_1, \ldots, X_n] \otimes \Lambda (\epsilon_1, \ldots, \epsilon_n)$, where $\Lambda(\epsilon_1, \ldots, \epsilon_n)$ denotes the exterior algebra over $\Z[\frac{1}{2}]$ in $n$ variables, and compatible with the prescribed action of $\mathcal{A}_p$ in the case $\F = \F_p$. 

\medskip
\noindent
Since $H^2(BT^{n+1},\F)$ can be identified with the root lattice tensored with $\F$, we pick polynomial generators of $H^*(BT^{n+1},\F)$ corresponding to the simple roots and order them so that the $i$-th simple root is denoted by $Y$, $\{ X_1, \ldots, X_{i-1}, Y, X_{i+1}, \ldots, X_n \}$. Using claim \ref{main}, and the definition of the link homology, the left hand side above is the cohomology of the cochain complex of the form $\partial_2 : Z_0 \longrightarrow Z_1$, with $Z_1$ being the complex:
\[ Z_1 = \F[X_1, \ldots, X_n, Y] \otimes \Lambda(\epsilon_1, \ldots, \epsilon_n, \epsilon). \]
The term $Z_0$ in the above complex is given by the homology of the Hochschild complex: 
\[ \F[X_1, \ldots, X_n, Y] \otimes_{\F[\hat{X}_1, \ldots, \hat{X}_n, Y^2]} \F[X'_1, \ldots, X'_n, Y'] \otimes \Lambda(\epsilon_1, \ldots, \epsilon_n, \epsilon), \quad \partial \epsilon_j = X_j - X'_j, \, \,  \partial \epsilon = Y - Y' \]
and where $\hat{X}_j = \frac{X_j + \sigma_i(X_j)}{2} = X_j \mod Y$. We may replace the generators $\epsilon_j$ by $\hat{\epsilon}_j$ so that $\partial{\hat{\epsilon}_j} = \hat{X}_j - \hat{X}'_j$ and $\hat{\epsilon}_j = \epsilon_j \mod \epsilon$. Notice that the classes $\hat{\epsilon}_j$ are permanent cycles in the Hochschild complex. It follows that $Z_0$ is the homology of the complex:
\[ \F[X_1, \ldots, X_n,Y] \otimes \Lambda(\hat{\epsilon}_1, \ldots, \hat{\epsilon}_n) \otimes_{\F[Y^2]} \F[Y'] \otimes \Lambda(\epsilon), \quad \partial(\epsilon) = Y-Y'. \]
This homology can easily be calculated to be a sum of two terms: 
\[ \F[X_1, \ldots, X_n, Y] \otimes \Lambda(\hat{\epsilon}_1, \ldots, \hat{\epsilon_n}) \, \, \bigoplus \, \, \F[X_1, \ldots, X_n,Y] \otimes \Lambda(\hat{\epsilon}_1, \ldots, \hat{\epsilon}_n) (Y+Y') \epsilon. \]
Now, the external differential is given by $\partial_2 (Y) = \partial_2(Y') = Y$. It follows from this that the homology of the external differential $\partial_2$ is only nontrivial in external degree $\bullet = 1$, and is given by: 
\[ H_1(\partial_2) = \frac{\F[X_1, \ldots, X_n, Y] \otimes \Lambda (\hat{\epsilon_1}, \ldots, \hat{\epsilon}_n) \, \epsilon}{\langle \, Y \, \rangle} = \F[X_1, \ldots, X_n] \otimes \Lambda (\epsilon_1, \ldots, \epsilon_n) \langle \epsilon \rangle. \]
Since all reduced Steenrod powers raise powers of $Y$, they act trivially on $\epsilon$. Consequently, if $\F$ is a field of odd characteristic, then the above isomorphism identifies the groups $H_*(HH_*(F_\bullet(\sigma_i), H^*(BT^{n+1},\F))$ canonically with the groups $H_*(HH_*(F_\bullet({\bf{1}}_{n+1}), H^*(BT^n,\F))$ compatibly under the action of the reduced Steenrod algebra $\mathcal{A}_p$ if one works over $\F_p$. 
\end{proof}
\section{Appendix} \label{Appendix}

\medskip
\noindent
In this appendix we describe a topological analog of the fundamental lemma of Rouquier \cite[Lemma 9.1]{R}. As in section \ref{Bott-Sam}, we work over the integers and all cohomology is assumed to be with coefficients in $\Z$ unless indicated otherwise. 

\medskip
\noindent
Let $A = (a_{ij})$ be the Cartan matrix of $G$. We will deduce, based on these results, that the Bott-Samelson map $\mbox{BS}_\bullet(\nu, n)$ is a homotopy equivalence away from all primes $p \leq a_{ij}a_{ji}$ and $n \leq m_{ij}$. We begin by assuming the conventions of section \ref{Bott-Sam}, and we let $k$ be any integer such that $0 \leq k < m_{ij}-1$. 

\begin{lemma} Working over the integers $\Z$, there exist topologically induced decompositions of chain complexes of bimodules:  

\medskip
\noindent
(A) The complex obtained by tensoring $H^\ast(\mathcal{X}(\nu,k))$ with $F(\sigma_\nu)$ on the right, is isomorphic to a sum of the following two complexes:
\[ \mbox{Id} : H^\ast(\mathcal{X}(\nu,k)) \rightarrow H^\ast(\mathcal{X}(\nu,k)), \quad \quad H^\ast(\mathcal{X}(\nu,k)) \otimes \delta_\nu \rightarrow 0, \]
where $H^\ast(\mathcal{X}(\nu, k)) \otimes \delta_\nu$ is the ideal in $H^\ast(\mathcal{X}(\nu, k)) \otimes_{H^\ast(BT)} H^\ast(\mathcal{F}(\sigma_\nu))$ generated by the Schubert basis element $\delta_\nu$. It is easy to see that this ideal is a sub-(bi)module. 

\medskip
\noindent
(B) Also consider the chain complex induced by (right) tensoring the restriction map with $H^\ast(\mathcal{F}(\sigma_\nu))$: 
\[ \mathcal{C} : H^\ast(\mathcal{X}(-\nu, k+1)) \otimes_{H^\ast(BT)} H^\ast(\mathcal{F}(\sigma_\nu))  \longrightarrow H^\ast(\mathcal{X}(\nu, k))\otimes_{H^\ast(BT)} H^\ast(\mathcal{F}(\sigma_\nu)). \]
Then the Bott-Samelson maps $\mu^\ast$ induces a short exact sequence of chain complexes of bimodules: 
\[ \{ H^\ast(\mathcal{X}(\nu, k+2)) \rightarrow H^\ast(\mathcal{X}(\nu,k)) \} \longrightarrow \mathcal{C} \longrightarrow \{ H^\ast(\mathcal{X}(\nu, k)) \otimes \delta_\nu \rightarrow H^\ast(\mathcal{X}(\nu,k)) \otimes \delta_\nu \}, \]
where the kernel is the complex given by cellular restriction, and the cokernel is the trivial complex given by the identity map. 

\end{lemma}
\begin{proof}
Consider the (degenerate) Bott-Samelson map:
\[ \mu : \mathcal{X}(\nu,k,\nu) \longrightarrow \mathcal{X}(\nu,k). \]
Recall that $H^\ast(\mathcal{X}(\nu,k,\nu)) = H^\ast(\mathcal{X}(\nu,k)) \otimes_{H^\ast(BT)} H^\ast(\mathcal{F}(\sigma_\nu))$. Furthermore, the map $\mu$ admits a zero-section, whose kernel is given by $H^\ast(\mathcal{X}(\nu,k)) \otimes \delta_\nu$. The first part of the lemma follows. For the second part, consider the following commutative diagram, with the horizontal maps induced by cellular inclusion and the vertical maps by multiplication: 
\[
\xymatrix {
\mathcal{X}(\nu, k, \nu) \ar[d]^{\mu_L} \ar[r] & \mathcal{X}(-\nu, k+1,\nu) \ar[d]^{\mu_R} \\
 \mathcal{X}(\nu, k) \ar[r] & \mathcal{X}(\nu, k+2).}
\]
In cohomology, this induces a diagram with injective vertical maps: 
\[
\xymatrix {
H^\ast(\mathcal{X}(\nu, k+2)) \ar[r] \ar[d]^{\mu_R^\ast} & H^\ast(\mathcal{X}(\nu, k)) \ar[d]^{\mu_L^\ast} \\
H^\ast(\mathcal{X}(-\nu, k+1)) \otimes_{H^\ast(BT)} H^\ast(\mathcal{F}(\sigma_\nu)) \ar[r] & H^\ast(\mathcal{X}(\nu, k)) \otimes_{H^\ast(BT)} H^\ast(\mathcal{F}(\sigma_\nu)). }
\]
It is easy to see that the cokernel of $\mu_L^\ast$ and $\mu_R^\ast$ are isomorphic. Now, by the previous part, the cokernel of $\mu_L^\ast$ can be identified with the bimodule: $H^\ast(\mathcal{X}(\nu,k)) \otimes \delta_\nu$ as needed. 
\end{proof}

\medskip
\begin{lemma}
Assume that one has inverted all primes $p$ such that $p \leq a_{ij}a_{ji}$, then the extension of the complex $\mathcal{C}$ described in part (B) of the previous lemma, splits as complexes of bimodules. 
\end{lemma}

\begin{proof}
By the proof of the previous lemma, the cokernel of the maps $\mu_L^*$ and $\mu_R^*$ agree with the bimodule $H^\ast(\mathcal{X}(\nu,k)) \otimes \delta_\nu$. Hence, to construt a splitting to $\mathcal{C}$, it is sufficient to establish a splitting to $\mu_R^\ast$ as bimodues. Recall that we defined the space $Y(\nu, k)$ to be the image of the Schubert variety $Z(\nu, k)$ in the homogeneous space $P/G_{\nu}$. Let $\mathcal{Y}(\nu,k)$ denote the space $EG \times_T Y(\nu,k)$. Recall also the pair of pullback diagrams considered earlier: 
\[
\xymatrix{
 \mathcal{X}(-\nu, k+1, \nu)  \ar[d]^{\pi} \ar[r]^{\quad \mu_R} & \mathcal{X}(\nu,k+2) \ar[d]^{\pi_\nu} \ar[r]^{\quad \rho} & EG/T \ar[d]^j \\
 \mathcal{X}(-\nu,k+1) \ar[r]^{\kappa} & \mathcal{Y}(\nu,k+2) \ar[r] & EG/G_\nu. 
}
\]
\smallskip
By the Eilenberg-Moore spectral sequence we notice that one gets: 
\[ H^\ast( \mathcal{X}(-\nu, k+1, \nu)) = H^\ast(\mathcal{X}(-\nu, k+1)) \otimes_{H^\ast(\mathcal{Y}(\nu,k+2))} H^\ast(\mathcal{X}(\nu, k+2)). \]
In order to construct a splitting to $\mu_R^\ast$ as bimodues, it is sufficient to construct a splitting of the map $\kappa$ in cohomology, in the category of modules over $H^\ast(BT) \otimes H^\ast(\mathcal{Y}(\nu,k+2))$. 

\smallskip
\noindent
Let us work in the cohomology ring $H^\ast(\mathcal{X}(\nu, k+2))$. Recall that by \cite[11.3.17]{Ku} there exists constants $a, d$ and elements $b, c \in H^2(BT)$ that satisfy the following relation among the Schubert basis elements in this ring: 
\[ \delta_{-\nu, k} \, \delta_\nu =  \delta_{\nu, k+1} - b \, \delta_{-\nu, k} + d \, \delta_{-\nu, k+1}, \quad \quad \delta_{-\nu} \, \delta_{-\nu, k} = a \, \delta_{-\nu, k+1} + c \, \delta_{-\nu, k}.  \]
Furthermore, the constant $a$ is non-zero, and its prime divisors $p$ are no bigger than $a_{ij}a_{ji}$. This can easily be checked by reducing to non-equivariant cohomology (see \cite[\S 10]{N}). 

\smallskip
\noindent
Now consider the class $e \in H^2(\mathcal{X}(\nu, k+2))$ given by: 
\[ e := \delta_{\nu} + y \, \delta_{-\nu} + x, \quad \quad y = -\frac{d}{a} \quad \quad x = \frac{ab + cd}{a}. \] 
We claim that the restriction to $H^\ast(\mathcal{X}(-\nu, k+1))$ of the $H^\ast(BT) \otimes H^\ast(\mathcal{Y}(\nu,k+2))$-submodule generated by the class $e$ is the complementary summand generating the splitting we seek. To prove this, consider the Serre spectral sequence for the fibration $\pi_\nu$ above. It follows from this spectral sequence that $H^\ast(\mathcal{X}(\nu, k+2))$ is a rank two free $H^\ast(BT) \otimes H^\ast(\mathcal{Y}(\nu, k+2))$-module on the classes $\{ 1, e \}$.  

\smallskip
\noindent
Let us consider the action on the class $e$ of the individual Schubert basis elements $\delta_{-\nu, i}$ $i \leq k+1$ that generate $H^\ast(\mathcal{Y}(\nu, k+2))$. It is easy to see that the classes $\{ e \, \delta_{-\nu, i}, 0 \leq i < k \}$ restrict to a free $H^\ast(BT))$-submodule in $H^\ast(\mathcal{X}(-\nu, k+1))$ that is a complement to the image of $H^\ast(\mathcal{Y}(\nu, k+2))$ as an $H^\ast(BT)$-submodule. It remains to show that this sub-module is closed under the action of $H^\ast(\mathcal{Y}(\nu, k+2))$. 

\smallskip
\noindent
Now, by the choice of our constants, we observe that $e \, \delta_{-\nu,k} = \delta_{\nu, k+1}$. In particular, $e \, \delta_{-\nu, k}$ is zero in $H^\ast(\mathcal{X}(-\nu,k+1))$, since $\delta_{\nu,k+1}$ restricts to zero in this ring. Recalling the relation $\delta_{-\nu} \, \delta_{-\nu, k} = a \, \delta_{-\nu, k+1} + c \, \delta_{-\nu, k}$, with $a \neq 0$, it also follows that $e \, \delta_{-\nu, k+1} = 0$. 

\smallskip
\noindent
The above argument shows that the $H^\ast(BT) \otimes H^\ast(\mathcal{Y}(\nu, k+2))$-submodule generated by the class $e$ is a summand in $H^\ast(\mathcal{X}(-\nu, k+1))$ that is a complement to $H^\ast(\mathcal{Y}(\nu, k+2))$. This is what we wanted to prove.

\end{proof}

\begin{remark}
Via the above splitting, and using the argument in (see \cite[proof of Prop. 9.2]{R}) it now follows that away from primes $p$ no bigger than the off-diagonal products in the Cartan matrix $a_{ij}a_{ji}$, the Bott-Samelson map $\mbox{BS}_\bullet(\nu,n)$ is a homotopy equivalence. So for example, in the case $G = PU(r+1)$, the condition is vacuous. For an arbitrary compact Lie group $G$, the only possible problematic primes are $p=2$ or $p=3$. This condition may be more restrictive if we were to consider Kac-Moody groups. 
\end{remark}

\pagestyle{empty}
\bibliographystyle{amsplain}
\providecommand{\bysame}{\leavevmode\hbox
to3em{\hrulefill}\thinspace}

\end{document}